\documentclass[11pt,a4paper]{amsart}

\usepackage[margin=1in]{geometry}
\usepackage{amsmath,amssymb,mathtools}
\usepackage{microtype}
\usepackage[colorlinks=true,linkcolor=blue,citecolor=blue,urlcolor=blue]{hyperref}

\newtheorem{theorem}{Theorem}[section]
\newtheorem{proposition}[theorem]{Proposition}
\newtheorem{lemma}[theorem]{Lemma}
\theoremstyle{definition}

\theoremstyle{remark}

\newcommand{\N}{\mathbb N}

\newcommand{\prodset}[1]{\mathop{\prod}\limits_{e\in #1}e}

\title[Distinct consecutive products]{Distinct Consecutive Products}
\author{Przemek Chojecki}
\date{\today. The proof was found by GPT-5.6 Sol, while iterating on previous attempts by the author.}

\begin{document}

\begin{abstract}
We prove that there is a density-one set of positive integers for which the products of all distinct consecutive blocks are distinct, answering a question of Erd\H{o}s and Graham.  The construction retains or rejects entire interiors of consecutive-prime gaps and is therefore stable under passage to longer prefixes.  Uniform affine-curve estimates control the witnesses with no previous rejected gap, and a scale-contracting forest controls every remaining witness unconditionally.
\end{abstract}

\maketitle

\section{Introduction}

Erd\H{o}s and Graham asked whether there is an increasing sequence $d_1<d_2<\cdots$ of density one such that the map
\[
 (u,v)\longmapsto \prod_{i=u}^{v}d_i,\qquad 1\leq u\leq v,
\]
is injective; see \cite[p.~84]{ErdosGraham}.  We answer the question affirmatively.

\begin{theorem}\label{thm:main}
There is a set $A\subseteq\N$ of natural density one such that distinct consecutive blocks in the increasing enumeration of $A$ have distinct products.
\end{theorem}

Two uniform results are used.  The first is the affine-curve estimate of Castryck--Cluckers--Dittmann--Nguyen \cite[Theorem~3]{CCDN}: there is an absolute $c>0$ such that, for every geometrically integral affine plane curve $C/\mathbb Q$ of degree $d$ and every $B\geq2$,
\begin{equation}\label{eq:CCDN}
 \#\bigl(C(\mathbb Z)\cap[-B,B]^2\bigr)
 \leq c d^3 B^{1/d}(\log B+d).
\end{equation}
The dependence is uniform in the coefficients of $C$.  The second is Li's theorem on primes in almost all short intervals \cite[Theorem~1.1]{Li}.  Fix a sufficiently large constant $D_0>2$ and then a sufficiently small $\varepsilon>0$.  That theorem permits us to arrange $\eta:=2/43+\varepsilon<1/20$ so that all but $O(X(\log X)^{-D_0})$ integers $n\in[X,2X]$ have a prime in $[n,n+n^\eta]$.

Put throughout
\begin{equation}\label{eq:parameters}
 \theta=\frac1{20},\qquad \rho=\frac1{2-\theta}=\frac{20}{39}.
\end{equation}
All implied constants may depend on the fixed $D_0$ and $\varepsilon$ and are otherwise absolute.  Expressions $X^{o(1)}$ are uniform over the dyadic scales occurring below.

\section{The construction and its witnesses}

We begin at the prime $2$ with $A_2=\{2\}$.  Suppose that $p<q$ are consecutive primes and that the construction has reached $p$.  Test the tentative prefix
\[
 B=A_p\cup\{p+1,p+2,\ldots,q\}.
\]
If two distinct consecutive blocks of $B$ have the same product, reject the gap and set $A_q=A_p\cup\{q\}$; otherwise retain the gap and set $A_q=B$.  Finally let $A=\bigcup_p A_p$.  Thus every prime is retained and, for each prime gap, either every or no integer in its interior is retained.

If two equal block products overlap, cancellation either leaves two disjoint equal block products or leaves a nonempty product equal to $1$.  The latter is impossible because all our integers are at least $2$.  We may consequently discuss only separated blocks, the earlier one being written $E$ and the later one $R$.

\begin{lemma}[Stability and canonical witnesses]\label{lem:stability}
Every finite prefix $A_p$ is collision-free, and so is $A$.  If the gap $(P,Q)$ is rejected, it admits a witness
\begin{equation}\label{eq:witness}
 \prodset{E}=\prod_{t=m}^{n}t,
\end{equation}
where $[m,n]\subset(P,Q)$ and $E$ is either a consecutive block of the established prefix or a full numerical interval in $(P,Q)$.  If $\ell=|E|$ and $s=n-m+1$, then $\ell>s$.
\end{lemma}

\begin{proof}
Assume inductively that $A_p$ is collision-free.  Any collision appearing in the tentative prefix is new.  After cancellation, its later block $R$ cannot contain a prime $r$, since then $r$ would divide a product of integers all strictly smaller than $r$.  If $R$ lay in $A_p$, the earlier block would do so as well, contradicting the induction; if $R$ met both $A_p$ and the newly adjoined elements, it would contain $p$.  Thus $R$ lies in the newly adjoined interior $(p,q)$ and is a full numerical interval $[m,n]$.

The earlier block $E$ cannot meet both $A_p$ and $(p,q)$.  Otherwise it contains $p$, so $p$ divides the right side of \eqref{eq:witness}; but Bertrand's postulate gives $q<2p$, and no integer of $(p,q)$ is divisible by $p$.  Thus $E$ has one of the asserted forms.  Since every element of $E$ is smaller than every element of $R$, equality of the products is impossible when $\ell\leq s$.

If the tentative prefix is accepted, it is collision-free by the test.  If it is rejected, $A_q=A_p\cup\{q\}$ is collision-free: a new collision would have a later block containing the terminal prime $q$, making $q$ divide a product of smaller integers.  This proves the induction.  A collision in $A$ would occur in a finite prefix, so none exists.
\end{proof}

Call a rejected gap \emph{raw} if it has a witness \eqref{eq:witness} in which $E$ crosses no earlier rejected gap; in this terminology a block crosses $(p,q)$ when it contains both $p$ and $q$.  A raw $E$ is a full interval: the only missing integers of the established prefix lie strictly between the endpoints of rejected prime gaps.  For every nonraw rejected gap, choose one witness and one earlier rejected gap crossed by $E$, and call the latter its \emph{parent}.  A parent has smaller right endpoint than its child, so this turns the rejected gaps into a directed forest.

\begin{lemma}[Parent--child alternatives]\label{lem:edge}
Let $(p,q)$, of length $g=q-p$, be the chosen parent of a gap rejected by the witness \eqref{eq:witness}.  There are integers $\alpha,\beta\geq2$ such that $\alpha p,\beta q\in[m,n]$.  Exactly one of the following alternatives may be chosen:
\begin{enumerate}
 \item[(E)] for some $j\geq2$, both $jp$ and $jq$ belong to $[m,n]$;
 \item[(U)] no such $j$ exists, and $p^2\leq ng+ps$.
\end{enumerate}
An edge of type \emph{(E)} or \emph{(U)} will be called equal or unequal, respectively.
\end{lemma}

\begin{proof}
The two boundary primes $p,q$ occur in $E$, hence divide the right side of \eqref{eq:witness}.  Their multiples in $[m,n]$ have multipliers at least $2$, since $q<m$.  If an equal multiplier occurs, choose (E).  Otherwise $|\alpha p-\beta q|\leq s-1$.  When $\alpha>\beta$, using $q=p+g$ and $\beta\leq n/p$ gives $p\leq\beta g+s\leq ng/p+s$; when $\alpha<\beta$, the same inequality gives $p\leq s$.  Both conclusions imply the inequality in (U).
\end{proof}

For a rejected gap $v=(p_v,q_v)$, write $x(v)=q_v$ and $d(v)=q_v-p_v$.  It is \emph{short} if $d(v)\leq p_v^\theta$ and \emph{long} otherwise.  Bounded gaps and endpoints can always be absorbed into the implied constants.

\begin{lemma}[Long gaps have negligible total length]\label{lem:long-tail}
For all large $Y$,
\[
 \sum_{\substack{Y\leq p<2Y\\p^+-p>p^\theta}}(p^+-p)
 \ll Y(\log Y)^{-D_0}.
\]
Here $p^+$ denotes the prime following $p$.  Consequently, the sum of the lengths of all long prime gaps with left endpoint at most $X$ is $o(X)$.
\end{lemma}

\begin{proof}
Use the fixed $\eta$ supplied above.  By Bertrand's postulate, a gap $(p,p^+)$ with $Y\leq p<2Y$ is contained in $[Y,4Y]$.  If its length $g$ exceeds $p^\theta$, then all but $O(Y^\eta+1)$ integers $a\in(p,p^+)$ have $[a,a+a^\eta]\subset(p,p^+)$.  For large $Y$ this supplies at least $g/2$ exceptional starting points for Li's theorem.  The sets of starting points supplied by distinct prime gaps are disjoint.  Applying that theorem on $[Y,2Y]$ and $[2Y,4Y]$ proves the displayed estimate.  In the dyadic summation, gaps whose left endpoints are below $X^{1/2}$ contribute $O(X^{1/2})$ by disjointness; on the remaining scales $\log Y\asymp\log X$.  This proves the final assertion.
\end{proof}

\section{Uniform counting of raw witnesses}

We first record the component-degree observation that makes \eqref{eq:CCDN} uniform for every split-product pattern.  It also avoids any irreducibility or low-genus classification.

\begin{proposition}[Split-product curves]\label{prop:split}
Let $r>s\geq1$ and
\[
 f(X)=\prod_{i=1}^{r}(X-a_i),\qquad
 g(Y)=\prod_{j=1}^{s}(Y+b_j),\qquad F(X,Y)=f(X)-g(Y),
\]
where the shifts are integers.  Put $h=(r,s)$ and $r=ph$, $s=qh$.  Every absolute irreducible component of $F=0$ has total degree $pe$ for some integer $e\geq1$.  In particular, uniformly in the shifts and for $B\geq2$,
\begin{equation}\label{eq:split-bound}
 \#\{(x,y)\in[-B,B]^2\cap\mathbb Z^2:F(x,y)=0\}
 \ll B^{1/p}(r^3\log B+r^4)
 \ll B^{1/2}(r^3\log B+r^4).
\end{equation}
\end{proposition}

\begin{proof}
Give $X$ weight $q$ and $Y$ weight $p$.  The highest weighted part of $F$ is
\[
 X^r-Y^s=\prod_{\zeta^h=1}(X^p-\zeta Y^q).
\]
The displayed factors are distinct and irreducible over $\overline{\mathbb Q}$ because $(p,q)=1$.  Multiplicativity of weighted initial forms and the squarefreeness of $X^r-Y^s$ show that the weighted initial form of an absolute factor $H$ of $F$ is, up to a constant, the product of $e\geq1$ of them.  It contains the monomials $X^{pe}$ and $Y^{qe}$, while every monomial $X^iY^j$ of $H$ satisfies $qi+pj\leq pqe$.  Since $p>q$, this implies $i+j\leq pe$, and therefore $\deg H=pe\geq p\geq2$.

Apply \eqref{eq:CCDN} to the absolute components defined over $\mathbb Q$.  Absolute components not defined over $\mathbb Q$ contribute only $O(r^2)$ rational points: such a point lies also on a distinct Galois-conjugate component, and B\'ezout's theorem applies.  If $\delta_\nu$ ranges over all absolute component degrees, then $\delta_\nu\geq p$, $\sum_\nu\delta_\nu\leq r$, $\sum_\nu\delta_\nu^3\leq r^3$, and $\sum_\nu\delta_\nu^4\leq r^4$.  Summing \eqref{eq:CCDN} proves \eqref{eq:split-bound}.
\end{proof}

\begin{lemma}[Raw gaps]\label{lem:raw}
The number of short raw rejected gaps with right endpoint at most $X$ is $O(X^{4/5+o(1)})$.
\end{lemma}

\begin{proof}
Choose a raw witness.  Both its sides are full intervals, so for some $r>s\geq1$ and positive $x,y\leq X$ it has the form
\begin{equation}\label{eq:factorial-curve}
 \prod_{i=0}^{r-1}(x-i)=\prod_{j=0}^{s-1}(y+j).
\end{equation}
Since the rejected gap is short, $s\leq X^\theta$.  Moreover $2^r\leq X^s$, whence $r\leq L:=\lceil X^\theta\log_2 X\rceil$.  For fixed $(r,s)$, Proposition~\ref{prop:split} bounds the possible $(x,y)$ by $O(X^{1/2}(r^3\log X+r^4))$.  Summing over $s<r\leq L$ gives
\[
 O\bigl(X^{1/2}(L^5\log X+L^6)\bigr)=O(X^{1/2+6\theta+o(1)})=O(X^{4/5+o(1)}).
\]
The interval on the right of \eqref{eq:factorial-curve} lies in at most one prime gap, so the same bound holds for gaps.
\end{proof}

\section{The witness forest}

We now count all short rejected gaps.  The loss in the following elementary out-degree estimate is harmless because unequal edges contract scale.

\begin{lemma}[Branches and contraction]\label{lem:branches}
For $Z\geq2$ the following assertions hold, with factors $Z^{o(1)}$ allowed.
\begin{enumerate}
 \item A fixed rejected gap has at most $Z^{3\theta+o(1)}$ short children with right endpoint at most $Z$.
 \item If an unequal edge has a short parent and a short child with right endpoint at most $Z$, then the parent's right endpoint is $O(Z^\rho)$.
 \item The number of long gaps that can be the parent, by an unequal edge, of a short child with right endpoint at most $Z$ is $O(Z^{1-\rho+o(1)})$.
 \item An equal edge cannot have a long parent and a short child, apart from bounded endpoints.
\end{enumerate}
\end{lemma}

\begin{proof}
For a short child, the later interval in \eqref{eq:witness} has $s\leq Z^\theta$; also $2^r\leq Z^s$, so $r=|E|\leq L_Z:=\lceil Z^\theta\log_2(2Z)\rceil$.  For each $r$, at most $r-1$ consecutive $r$-blocks cross a fixed parent adjacency.  Once such a block and $s$ are fixed, strict monotonicity of $m\mapsto\prod_{i=0}^{s-1}(m+i)$ determines at most one $m$, and its interval lies in at most one prime gap.  Hence there are at most $Z^\theta L_Z^2=Z^{3\theta+o(1)}$ children.

For an unequal edge, Lemma~\ref{lem:edge} gives $p^2\leq ng+ps$.  If the parent is short and the child lies below $Z$, then $g\leq p^\theta$, $n\leq Z$, and $s\leq Z^\theta$.  Thus $p^2\leq Zp^\theta+pZ^\theta$.  If $p\leq2Z^\theta$ the claim is immediate; otherwise the second term is absorbed, giving $p^{2-\theta}\ll Z$ and $p=O(Z^\rho)$.  Its right endpoint satisfies the same bound.

For the third assertion, note first that $p<Z$, and take parent left endpoints $p\asymp P$.  The case $P\ll Z^\theta$ is smaller than the claimed bound.  Otherwise the term $ps$ may be absorbed, and the unequal-edge inequality and longness give
\[
 g\gg\max(P^\theta,P^2/Z).
\]
Prime-gap interiors are disjoint, and Bertrand's postulate confines the gaps with $p\asymp P$ to an interval of length $O(P)$.  Their number is therefore at most
\[
 \min(P^{1-\theta},Z/P).
\]
The two terms agree at $P=Z^\rho$, where their common value is $Z^{\rho(1-\theta)}=Z^{1-\rho}$.  Summing over dyadic $P$ proves the assertion.

Finally, suppose that an equal edge with multiplier $j$ joins $(p,q)$ to the child gap $(P,Q)$.  Since $jp,jq\in(P,Q)$, we have $P<jp<jq<Q$ and $Q-P>jg$.  If the parent were long and the child short, then
\[
 Q-P>jg>jp^\theta=j^{1-\theta}(jp)^\theta>(jp)^\theta>P^\theta,
\]
contradicting $Q-P\leq P^\theta$.
\end{proof}

\begin{lemma}[Equal-chain compression]\label{lem:equal-chain}
Fix a rejected gap $(p,q)$ of length $g$.  There are $O(1+Z^\theta/g)$ short gaps of right endpoint at most $Z$ that can be reached from it by a nonempty path consisting only of equal edges.
\end{lemma}

\begin{proof}
If the successive equal multipliers are $j_1,\ldots,j_k$ and $J=j_1\cdots j_k$, induction shows that the terminal prime gap contains both $Jp$ and $Jq$: if $(P,Q)$ contains $Jp,Jq$, then an equal edge of multiplier $j$ has a full right interval containing $jP,jQ$, and hence also $jJp,jJq$.  For fixed $J$ there is at most one such prime gap.  Its length is at least $Jg$ and, by shortness, at most $Z^\theta$; hence $J\leq Z^\theta/g$.
\end{proof}

\begin{proposition}[Forest bound]\label{prop:forest}
The sum of the lengths of all short rejected gaps with right endpoint at most $X$ is $O(X^{9/10+o(1)})$.
\end{proposition}

\begin{proof}
Trace a short vertex backwards in the chosen-parent forest until reaching either a raw short gap or a long parent.  Compress every maximal string of equal edges.  By Lemmas~\ref{lem:branches} and \ref{lem:equal-chain}, at child scale $Z$ an unequal step has at most $Z^{3\theta+o(1)}$ choices and the following equal string has at most $Z^{\theta+o(1)}$ endpoints.  Backwards across a short unequal edge the scale falls from $Z$ to $O(Z^\rho)$.

First consider paths starting at a raw short gap and having exactly $t\geq0$ unequal edges.  Reading backwards from scale $X$, the successive relevant scales are at most $X,X^\rho,\ldots,X^{\rho^t}$, up to $X^{o(1)}$ factors.  Lemma~\ref{lem:raw}, the branch bounds, and the initial equal string show that the number of terminal gaps on such paths is at most $X^{B_t+o(1)}$, where
\[
 B_t=\frac45\rho^t+4\theta\sum_{i=0}^{t-1}\rho^i+\theta\rho^t.
\]
After multiplying by the maximum terminal length $X^\theta$, the exponent is
\begin{equation}\label{eq:raw-path-exponent}
 B_t+\theta
 =\frac45\rho^t+\frac{4\theta(1-\rho^t)}{1-\rho}
   +\theta\rho^t+\theta.
\end{equation}
It decreases with $t$ because the coefficient of $\rho^t$ after collecting terms is $4/5+\theta-4\theta/(1-\rho)=167/380>0$; at $t=0$ it equals $4/5+2\theta=9/10$.

For paths stopped at a long parent, let $t\geq1$ include the first long-to-short edge, which is unequal by Lemma~\ref{lem:branches}(4).  If its first short child has scale at most $X^{\rho^{t-1}+o(1)}$, part (3) of that lemma gives at most $X^{(1-\rho)\rho^{t-1}+o(1)}$ possible long parents.  The analogue of \eqref{eq:raw-path-exponent}, already including terminal lengths, is
\begin{equation}\label{eq:long-path-exponent}
 (1-\rho)\rho^{t-1}
 +\frac{4\theta(1-\rho^t)}{1-\rho}+\theta.
\end{equation}
This too decreases with $t$, since the coefficient of $\rho^{t-1}$ after collecting terms is $(1-\rho)-4\theta\rho/(1-\rho)=205/741>0$; at $t=1$ it is $1-\rho+5\theta=115/156<9/10$.

Repeated scale contraction gives $t=O(\log\log X)$.  There are only $O(\log X)$ choices for each intermediate dyadic scale.  Hence the total number of scale tuples is at most
\[
 (C\log X)^{C\log\log X}
 =\exp\bigl(O((\log\log X)^2)\bigr)=X^{o(1)},
\]
and the additional summation over $t$ has the same harmless size.  Equations \eqref{eq:raw-path-exponent} and \eqref{eq:long-path-exponent} prove the proposition.
\end{proof}

\begin{proof}[Proof of Theorem~\ref{thm:main}]
Every prime gap meeting $[1,X]$ has left endpoint at most $X$ and, by Bertrand's postulate, right endpoint at most $2X$.  Proposition~\ref{prop:forest}, applied at $2X$, therefore bounds the total length of the short rejected gaps meeting $[1,X]$ by $O(X^{9/10+o(1)})=o(X)$.  Lemma~\ref{lem:long-tail} gives the same conclusion for the long gaps.  Since the complement of $A$ consists, apart from $1$, precisely of the interiors of rejected prime gaps, we have $|[1,X]\setminus A|=o(X)$.  Thus $A$ has density one, and Lemma~\ref{lem:stability} gives the required injectivity of all consecutive-block products.
\end{proof}

\end{document}